\documentclass{article}
\usepackage{cite}
\usepackage{amsmath,amssymb,amsthm,mathrsfs}
\usepackage{booktabs}
\usepackage{titlesec,hyperref}
\usepackage{color}
\usepackage{soul}
\usepackage{fancyhdr}
\usepackage[margin=3cm]{geometry}
\usepackage{enumerate}
\usepackage[integrals]{wasysym}
\usepackage{bm}
\usepackage{graphicx}
\usepackage{caption}
\usepackage{float}
\usepackage{subfigure}

\newtheorem{theo}{Theorem}[section]
\newtheorem{lemm}{Lemma}[section]

\newtheorem{prop}{Proposition}[section]
\newtheorem{rema}{Remark}[section]
\numberwithin{equation}{section}

\def\ud{{\rm d}}

\allowdisplaybreaks

\begin{document}

\title{Ill-posedness and global solution for the $b$-equation}
\author{
Yingying $\mbox{Guo}^{1}$ \footnote{email: guoyy35@fosu.edu.cn}
\quad and\quad
Weikui $\mbox{Ye}^{2}$ \footnote{email: 904817751@qq.com}\\
$^1\mbox{School}$ of Mathematics and Big Data, Foshan University,\\
Foshan, 528000, China\\
$^2\mbox{School}$ of Mathematical Sciences, South China Normal University,\\
Guangzhou, 510631, China
}
\date{}
\maketitle
\begin{abstract}
In this paper, we consider the Cauchy problem for the $b$-equation. Firstly, for $s>\frac32,$ if $u_{0}(x)\in H^{s}(\mathbb{R})$ and $m_{0}(x)=u_{0}(x)-u_{0xx}(x)\in L^{1}(\mathbb{R}),$ the global solutions of the $b$-equation is established when $b\geq1$ or $b\leq1.$ It's worth noting that our global result is a new result which doesn't need the condition that $m_{0}(x)$ keeps its sign. For $s<\frac32,$ it is shown (see \cite{hgh}) that the Cauchy problem of the $b$-equation is ill-posed in Sobolev space $H^{s}(\mathbb{R})$ when $b>1$ or $b<1.$ In the present paper, for $s=\frac32,$ we prove  that the Cauchy problem of the $b$-equation is also ill-posed in $H^{\frac32}(\mathbb{R})$ in the sense of norm inflation by constructing a class of special initial data when $b\neq1.$ 
\end{abstract}
Mathematics Subject Classification: 35Q53, 35G25, 35D30\\
\noindent Keywords: $b$-equation, Global solutions,  Ill-posedness,Norm inflation.


\section{Introduction}
\par
In the paper, we consider the Cauchy problem for the following $b$-equation
\begin{equation}\label{b}
\left\{\begin{array}{ll}
m_t+um_{x}+bu_{x}m=0,&\qquad t>0,\ x\in\mathbb{R},\\
m_{0}(x)=m(0,x):=u(0,x)-u_{xx}(0,x),&\qquad x\in\mathbb{R}
\end{array}\right.
\end{equation}
where $u(t,x)$ stands for the fluid velocity and $m=u-u_{xx}$ denotes the momentum density.

Note that $(1-\partial_{xx})^{-1}f=p\ast f$ for any $f\in L^{2}(\mathbb{R}),$ where $\ast$ denotes the convolution and $p(x)=\frac12 e^{-|x|}.$ Then \eqref{b} can be rewritten as
\begin{equation}\label{bb}
\left\{\begin{array}{ll}
u_t+uu_{x}=-\partial_{x}(1-\partial_{xx})^{-1}\Big(\frac{b}{2}u^{2}+\frac{b-3}{2}u^{2}_{x}\Big),&\qquad t>0,\ x\in\mathbb{R},\\
u(0,x)=u_{0}(x),&\qquad x\in\mathbb{R}.
\end{array}\right.
\end{equation}

The $b$-equation was proposed by Holm and Staley \cite{hs1,hs2} as a one-dimensional version of active fluid transport that is described by the family of 1+1 evolutionary equations. In the study of soliton equations, it is found that Eq. \eqref{b} for any $b\neq-1$ is included in the family of shallow water equations at quadratic order accuracy that are asymptotically equivalent under Kodama transformations. On the other hand, it was shown by Degasperis and Procesi \cite{dep} that Eq. \eqref{b} cannot be completely integrable unless $b=2$ or $b=3$ by taking advantage of the method of asymptotic integrability. 

For $b=2$, Eq. \eqref{b} becomes the classical Camassa-Holm (CH) equation 
\begin{align}\tag{CH}
u_{t}-u_{txx}+3uu_{x}=2u_{x}u_{xx}+uu_{xxx}.
\end{align}
The CH equation has been studied extensively in the last two decades because of its many remarkable distinctive properties. One of the particular features about the CH equation is that it has single peakon solutions and multipeakon solutions. The local well-posedness and ill-posedness for the CH equation in Sobolev and Besov spaces were established in \cite{d2,lo,glmy}. Furthermore, it also has global strong solutions \cite{ce4,lo}. Another particular feature is the wave breaking phenomena (the solutions remain bounded but its slope becomes unbounded at infinite time) \cite{c2,lo}. In fact, wave breaking is the only way in which singularities can arise in a classical solution \cite{c2}. 

For $b=3$, Eq. \eqref{b} becomes the Degasperis-Procesi (DP) equation
\begin{align}\tag{DP}
u_{t}-u_{txx}+4uu_{x}=3u_{x}u_{xx}+uu_{xxx}.
\end{align}
The DP equation has a similar form to the CH equation and also has remarkable properties similar to the CH equation, such as infinite number of conservation laws and integrability,  but these two equations are quite different. It is worth noting that the DP equation has not only peakon solutions and periodic peakon solutions but also shock peakons and periodic shock waves \cite{dhh,ely2,lu}. On the other hand, the Cauchy problems of the DP equation has been studied extensively. The local well-posedness, ill-posedness, global strong solutions and several kinds of blow-up solutions were investigated in \cite{ck,ely2,gl,glmy,ly2,ly3,y5,y6}.

Although the $b$-family of equations are not  completely integrable for $b\neq0$ except for $b=2$ or $b=3,$ it has the following important conservation laws
\begin{align*}
H_{0}=\int_{\mathbb{R}}m\ud x,\quad H_{1}=\int_{\mathbb{R}}m^{\frac{1}{b}}\ud x,\quad H_{2}=\int_{\mathbb{R}}m_{x}m^{-2-\frac{1}{b}}+b^{2}m^{-\frac{1}{b}}\ud x.
\end{align*}
By the above conservation laws, most of the analytical properties of the $b$-family equation can be extracted from the CH and DP equations. Holm-Staley \cite{hs1,hs2} studied the peakon solutions of the $b$-equation for any $b\in\mathbb{R}.$ For $s>\frac32,$ Escher-Yin\cite{ey1} proved that the $b$-equation is locally well-posed in Sobolev spaces $H^{s}(\mathbb{R})$ and has global solutions for $0\leq b\leq1$ or $b=-\frac{1}{2n}.$ When $b>1,$ they also proved that the solution to \eqref{b} blows up in finite time if the initial data satisfies $m_{0}(x)\geq0,\ x\leq0$ or $m_{0}(x)\leq0,\ x\geq0.$ In \cite{glt}, Gui-Liu-Tian also obtained the blow-up phenonmena and global existence of solutions to \eqref{b} in $H^{s}(\mathbb{R})\big(s>\frac32\big)$ under various classes of initial data and the value of $b.$ Later, Lv-Wang\cite{lw} showed a new precise blow-up scenario and obtained some new blow-up results for the $b$-equation in $H^{s}(\mathbb{R})$ with $s>\frac32.$ However, for $s<\frac32,$ Himonas-Grayshan-Holliman \cite{hgh} studied the ill-posedness of the $b$-equations in $H^{s}(\mathbb{R}).$ Recently, for $s=\frac32,$ Guo-Liu-Molinet-Yin \cite{glmy} proved that the $b$-equation is ill-posed in critical Sobolev space $H^{\frac32}(\mathbb{R}).$ We summarize the above results in the following table: Table 1.
\begin{table}[htbp] 
\centering\caption{\label{tab}The global existence, blowup phenomena and ill-posedness for the $b-$equation}
\begin{tabular}{ccc}
\toprule
\text{The value of $b$}& Results & Reference \\
\midrule
\text{$0\leq b\leq1$} & Global solution & \cite{ey1} \\
\text{$b>1$} & Global solution, blow-up solution, ill-posedness & \cite{ey1}, \cite{glt}, \cite{lw}, \cite{hgh}, \cite{glmy} \\
\text{$b<-1$} & There are no results yet & ? \\
\text{$b=-\frac{1}{2n}$} & Global solution & \cite{ey1} \\
\bottomrule
\end{tabular}
\end{table}

In this paper, the first aim is to investigate the global existence of solutions for the $b$-equations to better understand the properties of the $b$-equations. We prove that the corresponding solution to \eqref{bb} is global in time with the initial data $u_{0}(x)\in H^{s}(\mathbb{R})(s>\frac32)$ and $ m_{0}(x)\in L^{1}(\mathbb{R})$ satisfying $m_{0}\geq0,x\leq x_{0}\big(\text{or}\ m_{0}\leq0,x\leq x_{0}\big)$ when $b\geq1$ or $m_{0}\geq0,x\geq x_{0}\big(\text{or}\ m_{0}\leq0,x\geq x_{0}\big)$ when $b\leq1.$ Comparing with the previous results in \cite{ey1,glt}, our result doesn't need the condition that $m_{0}(x)$ keeps its sign. This implies that our result is a new and interesting one.

For $b>1,$ it was shown (see \cite{hgh,glmy}) that the Cauchy problem of the $b$-equation is ill-posed in Sobolev space $H^{s}(\mathbb{R})$ for $s<\frac32$ or $s=\frac32$ by proving the norm inflation. However, for $b<-1,$ we see from Table 1 that there are no results about the $b$-equation yet.  Therefore, the second aim of this paper is to study the ill posedness of the $b$-equation in $H^{\frac32}(\mathbb{R})$ for $b<-1.$ 
Inspired by the idea of \cite{glmy}, we prove that the Cauchy problem of the $b$-equation is ill-posed in critical Besov space $B^{\frac{3}{2}}_{2,q}(\mathbb{R})$ with $1< q\leq\infty$ in the sense of norm inflation by constructing a class of special initial data for $b\neq1.$

Our main results are stated as follows.
\begin{theo}\label{global}
Suppose $u_{0}(x)\in H^{s}(\mathbb{R}),\ s>\frac{3}{2}$ and $m_{0}(x)=u_{0}(x)-u_{0xx}(x)\in L^{1}(\mathbb{R}).$ For $b\geq 1,$ if there is a $x_{0}\in\mathbb{R}$ such that
\begin{equation}\label{m}
\left\{\begin{array}{ll}
m_0(x)\geq0,&\qquad \text{if}\quad x\geq x_{0},\\
m_0(x)\leq0,&\qquad \text{if}\quad x\leq x_{0},
\end{array}\right.
\end{equation}
or for $b\leq1,$ if there exists a $x_{0}\in\mathbb{R}$ such that
\begin{equation}\label{m2}
\left\{\begin{array}{ll}
m_0(x)\geq0,&\qquad \text{if}\quad x\leq x_{0},\\
m_0(x)\leq0,&\qquad \text{if}\quad x\geq x_{0},
\end{array}\right.
\end{equation}
then the correspending solution to \eqref{bb} exists globally.
\end{theo}
\begin{rema}\label{rm1}
Theorem \ref{global} doesn't need the condition that the initial data $m_{0}(x)$ keeps its sign, which is a new result.
\end{rema}
\begin{theo}\label{ill}
Let $b\neq 1$ and $1<q\leq+\infty.$ For any $N\in\mathbb{N}^{+}$ large enough, there exists a $u_{0}\in \mathcal{C}^{\infty}(\mathbb{R})$ such that the following hold:

{\rm(1)} $\|u_{0}\|_{B^{\frac{3}{2}}_{2,q}}\leq\frac{1}{\ln N}\rightarrow0,\quad \text{as}\quad N\rightarrow +\infty;$

{\rm(2)} There is a unique solution $u\in \mathcal{C}\big([0,T);\mathcal{C}^{\infty}(\mathbb{R})\big)$ to the Cauchy problem \eqref{bb} with a time $T\leq\frac{1}{\ln N};$

{\rm(3)} $\|u\|_{L^{\infty}\big(0,T;B^{\frac{3}{2}}_{2,q}\big)}\geq\ln N.$
\end{theo}
\begin{rema}\label{rm2}
For $q=2,$ we see that Theorem \ref{ill} implies that the ill-posedness in the critical Sobolev space $H^{\frac32}.$
\end{rema}

The rest of our paper is as follows. In Section 2, we introduce some preliminaries which will be used in the sequel. In Section 3 and Section 4, we give the proof of our main results for the $b$-equation.

\section{Preliminaries}
\par
In this section, we first recall some basic properties on the Littlewood-Paley theory, which can be found in \cite{book}.

Let $\chi$ and $\varphi$ be a radical, smooth and valued in the interval $[0,1]$, belonging respectively to $\mathcal{D}(\mathfrak{B})$ and $\mathcal{D}(\mathfrak{C})$, where $\mathfrak{B}=\{\xi\in\mathbb{R}^d:|\xi|\leq\frac 4 3\},\ \mathfrak{C}=\{\xi\in\mathbb{R}^d:\frac 3 4\leq|\xi|\leq\frac 8 3\}$.
For all $f\in\mathcal{S}^{'},$ the Fourier transform $\mathcal{F}$ and its inverse $\mathcal{F}^{-1}$ are defined by
\begin{align*}
&\big(\mathcal{F}f\big)(\xi)=\hat{f}(\xi)=\int_{\mathbb{R}}e^{-ix\xi}f(x)\ud x,\\
&\big(\mathcal{F}^{-1}f\big)(x)=\check{f}(x)=\int_{\mathbb{R}}e^{-ix\xi}f(\xi)\ud\xi.
\end{align*}
The nonhomogeneous dyadic blocks $\Delta_j$ and low-frequency cut-off operators $S_j$ are defined as
\begin{align*}
&h=\mathcal{F}^{-1}\varphi,\quad \tilde{h}=\mathcal{F}^{-1}\chi,\\
&\Delta_j u=0\quad \text{if\quad $j\leq -2$,}\quad
\Delta_{-1} u=\chi(D)u=\int_{\mathbb{R}^{d}}\tilde{h}(y)u(x-y)\ud y,\\
&\Delta_j u=\varphi(2^{-j}D)u=2^{jd}\int_{\mathbb{R}^{d}}h(2^{jd}y)u(x-y)\ud y\quad \text{if\quad $j\geq0$,}\\
&S_j u=\sum\limits_{j'=-1}^{j}\Delta_{j'}u=\mathcal{F}^{-1}\Big(\chi\big(2^{-(j+1)}\xi\big)\mathcal{F}u\Big).
\end{align*}

Let $s\in\mathbb{R}$ and $1\leq p,\ r\leq\infty.$ The nonhomogeneous Besov space $B^s_{p,r}(\mathbb{R}^d)$ is defined by
\begin{align*}
B^s_{p,r}=B^s_{p,r}(\mathbb{R}^d)=\Big\{u\in \mathcal{S}'(\mathbb{R}^d):\|u\|_{B^s_{p,r}}=\big\|(2^{js}\|\Delta_j u\|_{L^p})_j \big\|_{l^r(\mathbb{Z})}<\infty\Big\}.
\end{align*}

The corresponding nonhomogeneous Sobolev space $H^{s}(\mathbb{R}^d)$ is
$$H^{s}=H^{s}(\mathbb{R}^d)=\Big\{u\in \mathcal{S}'(\mathbb{R}^d):\ u\in L^2_{loc}(\mathbb{R}^d),\ \|u\|^2_{H^s}=\int_{\mathbb{R}^d}(1+|\xi|^2)^s|\mathcal{F}u(\xi)|^2\ud\xi<\infty\Big\}.$$

We introduce some properties about Besov spaces. For more details, see \cite{book}.
\begin{prop}[See \cite{book}]\label{Besov}
Let $s\in\mathbb{R}$ and $1\leq p,\ r\leq\infty.$ \\
{\rm(1)} $B^s_{p,r}$ is a Banach space, and is continuously embedded in $\mathcal{S}'.$\\
{\rm(2)} For $p=2,\ r=2,$  the Besov space $B^{s}_{2,2}$ coincides with the Sobolev space $H^{s}.$\\
{\rm(3)} If $p_1\leq p_2$ and $r_1\leq r_2$, then $ B^s_{p_1,r_1}\hookrightarrow B^{s-(\frac 1 {p_1}-\frac 1 {p_2})}_{p_2,r_2}. $
If $s_1<s_2$, then the embedding $B^{s_2}_{p,r_2}\hookrightarrow B^{s_1}_{p,r_1}$ is locally compact.\\
{\rm(4)} If $r<\infty,$ then $\lim\limits_{j\rightarrow\infty}\|S_j u-u\|_{B^s_{p,r}}=0.$ If $p,\ r<\infty,$ then $\mathcal{C}_0^{\infty}$ is dense in $B^s_{p,r}.$\\
{\rm(5)} For $s\in\mathbb{R}^{+}\backslash\mathbb{N},$ the space $B^{s}_{\infty,\infty}$ coincides with the
H\"{o}lder space $\mathcal{C}^{[s],s-[s]};$ For $s\in\mathbb{N},$ the space 
$B^{s}_{\infty,\infty}$ is strictly
larger than the space $\mathcal{C}^{s}$ $($and than $\mathcal{C}^{s-1,1}$ if $s\in\mathbb{N}^{*}$$).$
\end{prop}

\section{The proof of Theorem \ref{global}}
\par
Before giving our global existence of the strong solutions to \eqref{bb}, let's first recall the local well-posedness and blow-up scenario of  the strong solutions.
\begin{lemm}[See \cite{ey1,glt}]\label{local}
Suppose $u_{0}\in H^{s}(\mathbb{R})$ with $s>\frac{3}{2}.$ Then there exist a maximal existence time $T=T(u_{0})>0$ and a unique solution $u$ to \eqref{bb} such that
\begin{align*}
u=u(\cdot,u_{0})\in \mathcal{C}\big([0,T];H^{s}\big)\cap \mathcal{C}^{1}\big([0,T];H^{s-1}\big)
\end{align*}
Moreover, the solution depends continuously on the initial data, that is the mapping $u_{0}\mapsto u: H^{s}\rightarrow \mathcal{C}\big([0,T];H^{s}\big)\cap \mathcal{C}^{1}\big([0,T];H^{s-1}\big)$ is continuous.
\end{lemm}
\begin{prop}[See \cite{ey1}]\label{en}
Let $u_{0}\in H^s(\mathbb{R})$ with $s>\frac 3 2.$ If $u \in C([0,T);H^s(\mathbb{R}))\cap C^1([0,T);H^{s-1}(\mathbb{R}))$ solves \eqref{bb}, then we have
\begin{align}\label{h}
\frac{\ud}{\ud t}\|u\|_{H^s}^2\leq C(\|u_x\|_{L^\infty}+1)\|u\|_{H^s}^2
\end{align}
for all $b\in\mathbb{R}.$
\end{prop}
\begin{lemm}\label{blowup}
Let $u_{0}\in H^{s}(\mathbb{R})$ with $s>\frac{3}{2}.$ Assume that $T>0$ is the maximal existence time of the corresponding solution $u$ to \eqref{bb} with the initial data $u_{0}.$ If $T$ is finite, then we have 
\begin{align*}
\int_{0}^{T}\|\partial_{x}u(t)\|_{L^{\infty}}\ud t<+\infty\quad\text{and}\quad\int_{0}^{T}\|u(t)\|_{B^{1}_{\infty,\infty}}\ud t<+\infty.
\end{align*} 
\end{lemm}
\begin{proof}
The proof of the lemma is similar to that of Theorem 3.25 in \cite{book} and here we omit it.
\end{proof}

Let's introduce the ordinary differential equation of the flow generated by the solution $u$ to \eqref{bb} with $u_{0}\in H^{s}(\mathbb{R})(s>\frac32)$
\begin{equation}\label{q}
\left\{\begin{array}{l}
\frac{\ud}{\ud t}y(t,x)=u(t,y(t,x)),\quad t\in[0,T),\quad x\in\mathbb{R},\\
y(0,x)=x,\qquad\qquad\quad\quad x\in\mathbb{R}.
\end{array}\right.
\end{equation}
By use of the classical theory of ordinary differential equations and the fact that $u \in H^s(\mathbb{R})$ with $s>\frac{3}{2},$ we see problem \eqref{q} has a unique solution $y\in\mathcal{C}^1\big([0,T)\times\mathbb{R};\mathbb{R}\big)$ such that the map $y(t,\cdot)$ is an increasing diffeomorphism of $\mathbb{R}$ with
\begin{equation}\label{yx}
y_x(t,x)=\exp\Big(\int_0^t u_x\big(t',y(t',x)\big)\ud t'\Big)>0,\quad \forall~(t,x)\in[0,T)\times\mathbb{R}.
\end{equation}
It follows from \eqref{yx} by setting $m(t,x)=u(t,x)-u_{txx}(t,x)$ that 
\begin{align}\label{mq}
m(t,y(t,x))y_{x}^{b}(t,x)=m_{0}(x),\quad \forall~(t,x)\in[0,T)\times\mathbb{R}.
\end{align}

\begin{proof}[\rm{\textbf{The proof of Theorem \ref{global}:}}] Applying the density argument, we only need to prove Theorem \ref{global} for $s=3.$ Due to Lemma \ref{local}, we know that \eqref{bb} has a unique local solution $u\in \mathcal{C}\big([0,T];H^{s}\big)\cap \mathcal{C}^{1}\big([0,T];H^{s-1}\big).$ Since $m=u-u_{xx}$ and the fact that $p(x)=\frac12e^{-|x|},$ we have
\begin{align}\label{u}
u(t,x)=\frac12e^{-x}\int_{-\infty}^{x}e^{\xi}m(t,\xi)\ud\xi+\frac12e^{-x}\int_{x}^{+\infty}e^{-\xi}m(t,\xi)\ud\xi.
\end{align} 
By differentiating \eqref{u} with respect to $x,$ we deduce
\begin{align}\label{ux}
u_{x}(t,x)=-\frac12e^{-x}\int_{-\infty}^{x}e^{\xi}m(t,\xi)\ud\xi+\frac12e^{-x}\int_{x}^{+\infty}e^{-\xi}m(t,\xi)\ud\xi.
\end{align}
Combing \eqref{u} and \eqref{ux}, we see
\begin{align}\label{uux}
\big(u^{2}-u_{x}^{2}\big)(t,x)=\int_{-\infty}^{x}e^{\xi}m(t,\xi)\ud\xi\int_{x}^{+\infty}e^{-\xi}m(t,\xi)\ud\xi.
\end{align}

For $b\leq1,$ using the assumption \eqref{m} in Theorem \ref{global} and the equality \eqref{yx}, we observe
\begin{equation}\label{mt}
\left\{\begin{array}{ll}
m\big(t,y(t,x)\big)\geq0,&\qquad \text{if}\quad x\leq y\big(t,x_{0}\big),\\
m\big(t,y(t,x)\big)\leq0,&\qquad \text{if}\quad x\geq y\big(t,x_{0}\big).
\end{array}\right.
\end{equation}

Similarly, for $b\geq1,$ we have 
\begin{equation}\label{mtt}
\left\{\begin{array}{ll}
m\big(t,y(t,x)\big)\geq0,&\qquad \text{if}\quad x\geq y\big(t,x_{0}\big),\\
m\big(t,y(t,x)\big)\leq0,&\qquad \text{if}\quad x\leq y\big(t,x_{0}\big).
\end{array}\right.
\end{equation}

Plugging \eqref{mt} or \eqref{mtt} into \eqref{uux}, we obtain
\begin{align}\label{uux0}
\big(u^{2}-u_{x}^{2}\big)\big(t,y(t,x_{0})\big)\leq0.
\end{align}  

Note that $m$ solves the following equation
\begin{align}\label{mm}
m_{t}+\big(um\big)_{x}=(1-b)u_{x}m.
\end{align}

Therefore, we have for $b\leq1$
\begin{align}
\|m\|_{L^{1}}=&\int_{-\infty}^{y(t,x_{0})}m\ud x-\int_{y(t,x_{0})}^{+\infty}m\ud x\notag\\
=&\|m_{0}\|_{L^{1}}+\int_{0}^{t}\bigg(\int_{-\infty}^{y(t,x_{0})}(1-b)u_{x}m\ud x+\int_{y(t,x_{0})}^{+\infty}-(1-b)u_{x}m\ud x\bigg)\ud\tau.\label{ml1}
\end{align}
and for $b\geq1$
\begin{align}
\|m\|_{L^{1}}=&\int_{y(t,x_{0})}^{+\infty}m\ud x-\int_{-\infty}^{y(t,x_{0})}m\ud x\notag\\
=&\|m_{0}\|_{L^{1}}+\int_{0}^{t}\bigg(\int_{y(t,x_{0})}^{+\infty}-(b-1)u_{x}m\ud x+\int_{-\infty}^{y(t,x_{0})}(b-1)u_{x}m\ud x\bigg)\ud\tau\label{ml11}
\end{align}

It then turns out from \eqref{uux0} that for $b\leq1$ 
\begin{align}
	\int_{-\infty}^{y(t,x_{0})}(1-b)u_{x}m\ud x=&\int_{-\infty}^{y(t,x_{0})}(1-b)u_{x}(u-u_{xx})\ud x\notag\\
	=&\frac{1-b}{2}\Big[u^{2}\big(t,y(t,x_{0})\big)-u_{x}^{2}\big(t,y(t,x_{0})\big)\Big]\leq0,\label{uxm1}\\
	\int_{y(t,x_{0})}^{+\infty}-(1-b)u_{x}m\ud x=&\int_{y(t,x_{0})}^{+\infty}-(1-b)u_{x}(u-u_{xx})\ud x\notag\\
	=&-\frac{1-b}{2}\Big[-u^{2}\big(t,y(t,x_{0})\big)+u_{x}^{2}\big(t,y(t,x_{0})\big)\Big]\notag\\
	=&\frac{1-b}{2}\Big[u^{2}\big(t,y(t,x_{0})\big)-u_{x}^{2}\big(t,y(t,x_{0})\big)\Big]\leq0.\label{uxm2}
\end{align}

The similar argument gives that for $b\geq1$
\begin{align}
	\int_{y(t,x_{0})}^{+\infty}-(b-1)u_{x}m\ud x=&\int_{y(t,x_{0})}^{+\infty}-(b-1)u_{x}(u-u_{xx})\ud x\notag\\
	=&-\frac{b-1}{2}\Big[-u^{2}\big(t,y(t,x_{0})\big)+u_{x}^{2}\big(t,y(t,x_{0})\big)\Big]\notag\\
	=&\frac{b-1}{2}\Big[u^{2}\big(t,y(t,x_{0})\big)-u_{x}^{2}\big(t,y(t,x_{0})\big)\Big]\leq0,\label{uxm3}\\
	\int_{-\infty}^{y(t,x_{0})}(b-1)u_{x}m\ud x=&\int_{y(t,x_{0})}^{+\infty}-(b-1)u_{x}(u-u_{xx})\ud x\notag\\
	=&\frac{b-1}{2}\Big[u^{2}\big(t,y(t,x_{0})\big)-u_{x}^{2}\big(t,y(t,x_{0})\big)\Big]\leq0.\label{uxm4}
\end{align}

In view of \eqref{uxm1}--\eqref{uxm4}, we see from \eqref{ml1} or \eqref{ml11} that
\begin{align}
\|m\|_{L^{1}}\leq \|m_{0}\|_{L^{1}}.\label{mm0}
\end{align}
Noting again that $u-u_{xx}=m,$ we get from the $L^{1}$-theory for linear elliptic equations that $u(t,\cdot)\in W^{2,1}(\mathbb{R}).$ Then the Sobolev imbedding theorem $W^{2,1}(\mathbb{R})\hookrightarrow \mathcal{C}_{B}^{1}(\mathbb{R})$ and the inequality \eqref{mm0} together imply that $\|u(t,\cdot)\|_{L^{\infty}}$ and $\|u_{x}(t,\cdot)\|_{L^{\infty}}$ are uniformly bounded for all $t\in [0,T).$ We conclude from Lemma \ref{blowup} that the corresponding solution to \eqref{bb} is global in time. This completes the proof of Theorem \ref{global}. 
\end{proof}

\section{The proof of Theorem \ref{ill}}
This section is devoting to establishing the ill-posedness for the problem \eqref{bb} by the norm inflation.

\begin{proof}[\rm{\textbf{The proof of Theorem \ref{ill}:}}] 
Let $\tilde{\varphi}\in\mathcal{D}$ is an even, non-negative, non-zero function such that $\tilde{\varphi}\varphi=\tilde{\varphi}.$ 

\textbf{Case 1: $b<1.$} For $N>0$ large enough, define $$u_{0}(x)=-\sum\limits_{n=2}^{N}\frac{h_n(x)\cdot(\ln N)^{-1}}{2^{2n}n^{\frac{2}{1+q}}}$$ 
where $\hat{h}_n(\xi)=i2^{-n}\xi\tilde{\varphi}(2^{-n}\xi).$ It is easy to veirfy that $u_{0}\in \mathcal{C}^{\infty}(\mathbb{R})\cap B^{\frac 3 2}_{2,q}(\mathbb{R})$ is odd and
\begin{align}
&\|u_{0}\|_{B^{\frac 3 2}_{2,q}}\leq\frac{1}{\ln N},\label{u0ln}\\
&u_{0x}(0)\geq N^{1-\frac{2}{1+q}}=N^{\frac{q-1}{1+q}}>0.\label{ux0}
\end{align}

Since $u_{0}\in H^{s}(\mathbb{R})$ with $s>\frac32,$ by taking advantage of Theorem \eqref{local}, we know that there exists a unique solution $u\in \mathcal{C}([0,T_N);H^\infty)$ for fixed $N>0.$

Assume that there exists a $T_{1}>0$ (~let's assume $T_{1}=\frac{1}{\ln N}$~) small enough such that 
\begin{align}
\|u\|_{L^{\infty}(0,T_{1};B^{\frac32}_{2,q})}\leq \ln N.\label{uln}
\end{align}
It follows from Theorem \ref{local} and Proposition \ref{en} that $u\in L^{\infty}\big(0,T_{1};H^{s}\big)$ and 
\begin{align*}
\|u\|_{L^{\infty}_{T_{1}}H^{s}}\leq \|u\|_{L^{\infty}_{T_{1}}H^{s}}e^{\int_{0}^{T_{1}}\|u\|_{B^{\frac32}_{2,\infty}}\ud t}.
\end{align*}

Noting that $u_{0}(x)$ is odd, we see by the uniqueness of solutions that $u(t,x)$ is odd and $u(t,0)=u_{xx}(t,0)=0$ on $[0,T_{1}].$ Differentiating Eq. \eqref{bb} with respect to $x,$ we deduce
\begin{align}
u_{tx}+uu_{xx}=\frac{1-b}{2}u_{x}^{2}+\frac{b}{2}u^{2}-(1-\partial_{xx})^{-1}\Big(\frac{b}{2}u^{2}+\frac{3-b}{2}u_{x}^{2}\Big).\label{uxt}
\end{align}
Setting $x=0$ in \eqref{uxt} and then applying the inequality \eqref{uln}, we get
\begin{align}\label{ut0}
\frac{\ud}{\ud t}u_{x}(t,0)\geq \frac{1-b}{2}u_{x}^{2}(t,0)-2\big(\ln N\big)^{2},\qquad \forall t\in[0,T_{1}).
\end{align}
Thanks to \eqref{ux0}, we see for $q>1$ and $N>0$ large enough that
\begin{align*}
u_{0x}(0)\geq N^{\frac{q-1}{1+q}}>\frac{2\ln N}{\sqrt{1-b}}>0.
\end{align*}
It follows that $u_{x}(t,0)>\frac{2\ln N}{\sqrt{1-b}}>0.$ Then solving inequality \eqref{ut0} gives 
\begin{align}
1-\frac{\frac{\sqrt{b-1}}{2}u_{0x(0)-\ln N}}{\frac{\sqrt{b-1}}{2}u_{0x(0)+\ln N}}e^{2\sqrt{2}\ln N t}\geq\frac{2\ln N}{\frac{\sqrt{b-1}}{2}u_{0x(0)+\ln N}}>0.
\end{align}
Owing to $0<\frac{\frac{\sqrt{b-1}}{2}u_{0x(0)-\ln N}}{\frac{\sqrt{b-1}}{2}u_{0x(0)+\ln N}}<1,$ we deduce that there exists $0<T_{2}<\frac{1}{2\sqrt{b-1}\ln N}\ln\frac{\frac{\sqrt{2}}{2}u_{0x(0)+\ln N}}{\frac{\sqrt{b-1}}{2}u_{0x(0)-\ln N}}<T_{1}$ for $N>0$ large enough such that the solution blows up
in finite time $T_{2}.$ By virtue of the blow-up scenario i.e. Lemma \ref{blowup}, we find 
\begin{align*}
\|u\|_{L^{\infty}(0,T_{1};B^{1}_{\infty,\infty})}=+\infty
\end{align*}
which is contradict to the assumption \eqref{uln}. Therefore, 
\begin{align*}
\|u\|_{L^{\infty}(0,T_{1};B^{\frac32}_{2,q})}\geq\ln N.
\end{align*}

\textbf{Case 2:} For $b>1,$ set $u_{0}(x)=\sum\limits_{n=2}^{N}\frac{h_n(x)\cdot(\ln N)^{-1}}{2^{2n}n^{\frac{2}{1+q}}}.$ Then the proof is similar to that of in \textbf{Case 1} and here we are not repeating it.

This finishes the proof of Theorem \ref{ill}.
\end{proof}

\noindent\textbf{Acknowledgements.}
Guo was supported by the National Natural Science Foundation of China (No. 12301298, No. 12161004), the Basic and Applied Basic Research Foundation of Guangdong Province (No.
2020A1515111092) and Research Fund of Guangdong-Hong Kong-Macao Joint Laboratory for Intelligent Micro-Nano Optoelectronic Technology (No. 2020B1212030010). Ye was supported by the general project of NSF
of Guangdong province (No. 2021A1515010296).


\end{document}